\documentclass{roalart}

\pdfoutput=1

\author[Roberto Albesiano]{Roberto Albesiano}
\address{Department of Pure Mathematics, University of Waterloo}
\email{roberto.albesiano@uwaterloo.ca}
\urladdr{https://math.uwaterloo.ca/~ralbesiano/}

\title{Generically surjective morphisms of holomorphic vector bundles via degenerations}

\date{\today}

\begin{document}

\begin{abstract}
    We prove an $L^2$ theorem on generically surjective morphism of holomorphic vector bundles via a degeneration argument, generalizing the author's previous work on the $L^2$ division theorem of Skoda.  The proof is based on Berndtsson's theorem on the positivity of direct image bundles and is inspired by Berndtsson and Lempert's proof of the $L^2$ extension theorem.
\end{abstract}

\section*{Introduction}
The purpose of this paper is to prove the following $L^2$ theorem on generically surjective holomorphic morphism of holomorphic vector bundles, generalizing the author's previous work on H.~Skoda's $L^2$ division theorem \cite{Albesiano2024}.

Let $X$ be a complex manifold.  If $h$ is a Hermitian metric for a holomorphic vector bundle $V \to X$ with curvature $\Theta_h$, we define the functions $\lambda_h, \Lambda_h : T_X^{1,0} \to \mathbb{R}$ by 
\[
    \lambda_h(\zeta) := \min_{v \in V_x} \frac{h((\Theta_h)_{\zeta\bar{\zeta}}v,\bar{v})}{h(v,\bar{v})}
    \quad\text{and}\quad
    \Lambda_h(\zeta) := \max_{v \in V_x} \frac{h((\Theta_h)_{\zeta\bar{\zeta}}v,\bar{v})}{h(v,\bar{v})}
\]
for all $\zeta \in T_{X,x}^{1,0}$.  To ease notation, we will also write $r_V := \rk V$.

\begin{themainthm}
    Let $X$ be a (essentially) Stein manifold, and let $F,G \to X$ be holomorphic vector bundles with Hermitian metrics $h_F, h_G$, respectively.  Let $F \overset{\gamma}{\longto} G$ be a generically surjective holomorphic morphism of vector bundles.  Fix $\alpha > 1$ and assume that
    \[
        (r_G + 1) \lambda_{h_G} + (r_F - 1)\Tr\Theta_{h_G} - r_G \Tr\Theta_{h_F} \geq (r_F r_G+1 + \alpha(r_F r_G - 1)) (\Lambda_{h_G} - \lambda_{h_F}).
    \]
    Then for any holomorphic section $g$ of $G \otimes K_X \to X$ such that 
    \[
        \norm{g}_G^2 := \int_X \frac{h_G(g,\bar{g})}{(h_F^* \otimes h_G)(\gamma,\bar{\gamma})^{\alpha(r_F r_G - 1)+1}} < +\infty
    \]
    there is a holomorphic section $f$ of $F \otimes K_X \to X$ such that $g = \gamma f$ and 
    \[
        \norm{f}_F^2 := \int_X \frac{h_F(f,\bar{f})}{(h_F^* \otimes h_G)(\gamma,\bar{\gamma})^{\alpha(r_F r_G - 1)}} \leq r_F \frac{\alpha}{\alpha-1} \norm{g}_G^2.
    \]
\end{themainthm}

The philosophy of proof is similar to B.~Berndtsson and L.~Lempert's proof of the $L^2$ extension theorem with sharp constants \cite{BerndtssonLempert2016,Lempert2017}: we write a family of norms that at one end gives the norm of the solution $\norm{f}_F$, and at the other end degenerates to the norm of the datum $\norm{g}_G$, up to a constant.  This family of norms is induced by a family of metrics for the hyperplane bundle over the projectivization of the dual of the bundle $F^* \otimes G \to X$ of morphisms from $F$ to $G$.  The assumptions on curvature together with the vector bundles version of Berndtsson's theorem on the positivity of direct image bundles \cite{Berndtsson2009,Varolin2025} give convexity in the family of norms, which in turns gives the estimate we are looking for.

\subsection*{Two special cases}
Let $L \to X$ be a holomorphic line bundle with Hermitian metric $\e^{-\phi}$.  Then $\gamma$ induces a generically surjective holomorphic morphism $F \otimes L \overset{\gamma}{\longto} G \otimes L$, and the condition on curvature of this twisted problem becomes 
\[
    \I\ddbar\phi + (r_G + 1) \lambda_{h_G} - \Tr\Theta_{h_G} - \Tr\Theta_{h_F \otimes h_G^*} \geq (r_F r_G+1 + \alpha(r_F r_G - 1)) (\Lambda_{h_G} - \lambda_{h_F}).
\]
If $h_G$ and $h_F \otimes h_G^*$ have non-negative curvature in the sense of Griffiths then $\lambda_{h_G} \geq 0$ and $\Lambda_{h_G} - \lambda_{h_F} \leq 0$.  Hence the \mainthmref has the following corollary.

\begin{maincorollary}\label{cor:Skoda}
    Let $X$ be a (essentially) Stein manifold, and let $F,G \to X$ be holomorphic vector bundles with Hermitian metrics $h_F, h_G$, respectively.  Let $F \overset{\gamma}{\longto} G$ be a generically surjective holomorphic morphism of vector bundles and fix $\alpha > 1$.  Assume that $h_G$ and $h_F \otimes h_G^*$ have non-negative curvature in the sense of Griffiths, and that there is a holomorphic line bundle $L \to X$ with Hermitian metric $\e^{-\phi}$ such that 
    \[
        \I\ddbar\phi \geq \Tr\Theta_{h_F \otimes h_G^*} + \Tr \Theta_{h_G}.
    \]
    Then for any holomorphic section $g$ of $G \otimes L \otimes K_X \to X$ such that 
    \[
        \norm{g}_{G \otimes L}^2 = \int_X \frac{h_G(g,\bar{g}) \e^{-\phi}}{(h_F^* \otimes h_G)(\gamma,\bar{\gamma})^{\alpha(r_F r_G - 1)+1}} < +\infty
    \]
    there is a holomorphic section $f$ of $F \otimes L \otimes K_X \to X$ such that $g = \gamma f$ and 
    \[
        \norm{f}_{F \otimes L}^2 = \int_X \frac{h_F(f,\bar{f}) \e^{-\phi}}{(h_F^* \otimes h_G)(\gamma,\bar{\gamma})^{\alpha(r_F r_G - 1)}} \leq r_F \frac{\alpha}{\alpha-1} \norm{g}_{G \otimes L}^2.
    \]
\end{maincorollary}

Compared to the Griffiths-curvature version \cite[Théorème 7]{DemaillySkoda1980} of Skoda's theorem \cite[Théorème 2]{Skoda1981} below (see also \cite[Theorem 11.8]{Demailly2001}), \autoref{cor:Skoda} requires the extra assumptions that the curvatures of $h_G$ and $h_F \otimes h_G^*$ are non-negative in the sense of Griffiths (thus in particular $\Theta_{h_F} \geqGrif 0$), but puts a different requirement on the curvature of $\e^{-\phi}$ and involves different norms.

\begin{theorem*}[\textbf{Demailly--Skoda}]
    Let $X$ be a Stein manifold, and let $F,G \rightarrow X$ be holomorphic vector bundles endowed with Hermitian metrics $h_F$ and $h_G$, respectively.  Let $F \overset{\gamma}{\longto} G$ be a generically surjective holomorphic morphism of vector bundles and let $L \to X$ be a line bundle with Hermitian metric $\e^{-\phi}$.  Set $\alpha>1$ and
    \[
      q := \min(r_F - r_G, \dim X).
    \]
    Assume that $\Theta_{h_F} \geqGrif 0$ and
    \[
      \I\ddbar\phi \geq \Tr \Theta_{h_F} + \alpha q \Tr \Theta_{h_G}.
    \]
    Then, for any holomorphic section $g \in H^0(X, G \otimes L \otimes K_X)$ such that
    \[
      \norm{g}_{G \otimes L}^2 := \int_X \frac{h_G((\gamma\gamma^*)^{-1} g,\bar{g})\e^{-\phi}}{\det(\gamma\gamma^*)^{\alpha q}} < +\infty,
    \]
    there is a holomorphic section $f \in H^0(X, F \otimes E \otimes K_X)$ such that $g = \gamma f$ and
    \[
      \norm{f}_{F \otimes L}^2 := \int_X \frac{h_F(f,f)\e^{-\phi}}{\det(\gamma\gamma^*)^{\alpha q}} \leq \frac{\alpha}{\alpha-1} \norm{g}_{G \otimes L}^2.
    \]
    Here $\gamma^*: G \to F$ is the adjoint of $\gamma$ with respect to $h_F$ and $h_G$.
\end{theorem*}

At least when $r_G = 1$, \autoref{cor:Skoda} suffers the same pathologies encountered in \cite{Albesiano2024}: the constant $r_F \frac{\alpha}{\alpha-1}$ should be $\frac{\alpha}{\alpha-1}$, and the quantity $r_F r_G - 1$ should be replaced by $q r_G$.  At the moment, it is not clear to the author whether these deficiencies are caused by a lack of sharpness in the argument below, or whether they are instead a feature of the degeneration technique used.  Either way, as in \cite{Albesiano2024} the main emphasis is on the technique involved in the proof of the \mainthmref, rather than the statement itself.

The second special case recovers the main result of \cite{Albesiano2024} without the upper bound on the parameter $\alpha > 1$.  Let $E \to X$ be a holomorphic line bundle with Hermitian metric $\e^{-\varphi}$, and assume that $G \to X$ is also of rank one, with metric $h_G = \e^{-\psi}$.  Fix $h_1,\dots,h_r$ holomorphic sections of $E^* \otimes G \to X$.  By taking $F := E^{\oplus r}$ and $h_F = \Id_F \e^{-\varphi}$ we have 
\[
\begin{split}
    &(r_G + 1)\lambda_{h_G} + (r_F - 1) \Tr\Theta_{h_G} - r_G \Tr\Theta_{h_F} - (r_F r_G+1 + \alpha(r_F r_G - 1)) (\Lambda_{h_G} - \lambda_{h_F}) \\
    &\quad= \I\ddbar \left[2\psi + (r-1) \psi - r\varphi - (r+1 + \alpha(r-1))(\psi - \varphi)\right] \\
    &\quad= \left(\alpha(r-1)+1\right) \I\ddbar\varphi - \alpha(r-1)\I\ddbar\psi.
\end{split}
\]
All together this gives the following line bundle version of the \mainthmref.
\begin{maincorollary}
    Let $X$ be a (essentially) Stein manifold, and let $E,G \to X$ be holomorphic line bundles with (singular) Hermitian metrics $\e^{-\varphi}, \e^{-\psi}$, respectively.  Fix $h_1,\dots,h_r$ holomorphic sections of $E^* \otimes G \to X$ and $\alpha>1$.  Assume that 
    \[
        \I\ddbar\varphi \geq \frac{\alpha(r-1)}{\alpha(r-1)+1} \I\ddbar\psi.
    \]
    Then for any holomorphic section $g$ of $G \otimes K_X \to X$ such that 
    \[
        \norm{g}_G^2 = \int_X \frac{|g|^2 \e^{-\psi}}{(|h|^2 \e^{-\psi+\varphi})^{\alpha(r-1)+1}}
    \]
    there are holomorphic sections $f_1,\dots,f_r$ of $E \otimes K_X \to X$ such that
    \[
        g = h_1 f_1 + \dots + h_r f_r
    \]
    and 
    \[
        \norm{f}_F^2 = \int_X \frac{|f|^2 \e^{-\varphi}}{(|h|^2 \e^{-\psi+\varphi})^{\alpha(r-1)}} \leq r \frac{\alpha}{\alpha-1} \norm{g}_G^2.
    \]
\end{maincorollary}

The removal of the artificial upper bound on $\alpha$ is achieved in the proof of the \mainthmref by more carefully computing the curvature of the weights involved in the degeneration argument (see \autoref{rmk:alphaUpperBound}).

\subsection*{Organization}
In \autoref{sec:preliminaries}, after some general preliminary observations about the measures of curvature $\lambda$ and $\Lambda$, we reduce the manifold to a bounded pseudoconvex domain where $\gamma$ is surjective, and prove that with such reductions there is a solution with finite $L^2$ norm.  We then restate the problem in a dual setting.  In \autoref{sec:family} we set up the degenerating family of norms that produces the estimate, and compute what happens at its extrema.  In \autoref{sec:curvature} we study the curvature of the family of norms, and finally in \autoref{sec:convexity} we conclude the proof of the \mainthmref by running the Berndtsson--Lempert argument on the family.

\subsection*{Acknowledgements}
I am grateful to Dror Varolin for helpful observations on a first draft of this paper.

\section{Preliminaries}\label{sec:preliminaries}
\subsection{Curvature}
Let $V \to X$ be a holomorphic vector bundle, and let $h$ be a Hermitian metric for $V$.  Fix a Hermitian metric $\mathfrak{g}$ on $X$.  The curvature $\Theta_h$ of the metric $h$ is Griffiths-positive at $x \in X$ if there is $c>0$ such that 
\[
    \frac{h((\Theta_h)_{\zeta\bar{\zeta}}v,\bar{v})}{h(v,\bar{v})} \geq c \mathfrak{g}(\zeta,\bar{\zeta})
\]
for all $v \in V_x \setminus \{0\}$ and $\zeta \in T_{X,x}^{1,0}$.  For any metric $h$ for $V \to X$ we define the functions of $\zeta \in T_{X,x}^{1,0}$ given by
\[
    \lambda_h(\zeta) := \min_{v \in V_x} \frac{h((\Theta_h)_{\zeta\bar{\zeta}}v,\bar{v})}{h(v,\bar{v})}
    \quad\text{and}\quad
    \Lambda_h(\zeta) := \max_{v \in V_x} \frac{h((\Theta_h)_{\zeta\bar{\zeta}}v,\bar{v})}{h(v,\bar{v})}.
\]
In particular, $h$ is Griffiths-positive at $x \in X$ if $\lambda_h(\zeta) > 0$ for all $\zeta \in T_{X,x}^{1,0}$, and Griffiths-negative at $x \in X$ if $\Lambda_h(\zeta) < 0$ for all $\zeta \in T_{X,x}^{1,0}$.

Notice that the ratio \(h((\Theta_h)_{\zeta\bar{\zeta}}v,\bar{v})/h(v,\bar{v})\) is invariant under scaling of $v$ by non-zero complex numbers, and so it is a well-defined function on $P(V_x)$.

Since
\[
    h^*((\Theta_{h^*})_{\zeta\bar{\zeta}}v^*,\bar{v}^*) = -h((\Theta_h)_{\zeta\bar{\zeta}}v,\bar{v}),
\]
we immediately obtain
\[
    \lambda_{h^*} = -\Lambda_h \quad \text{and} \quad \Lambda_{h^*} = -\lambda_h.
\]

If $V_1, V_2 \to X$ are holomorphic vector bundles with Hermitian metrics $h_1,h_2$, respectively, one has
\[
    \Theta_{h_1 \otimes h_2} = \Theta_{h_1} \otimes \Id_{V_2} + \Id_{V_1} \otimes \Theta_{h_2},
\]
and so by \autoref{lemma:linearAlgebra} below we conclude that
\[
    \lambda_{h_1 \otimes h_2} = \lambda_{h_1} + \lambda_{h_2} \quad \text{and} \quad \Lambda_{h_1 \otimes h_2} = \Lambda_{h_1} + \Lambda_{h_2}.
\]

\begin{lemma}\label{lemma:linearAlgebra}
    Let $A_1: W_1 \to W_1$ and $A_2: W_2 \to W_2$ be Hermitian matrices and set 
    \[
    A := A_1 \otimes \Id_{W_2} + \Id_{W_1} \otimes A_2.
    \]
    Then the spectrum of $A$ is
    \[
        \spec(A) = \spec(A_1) + \spec(A_2) := \left\{ \mu_1 + \mu_2 \, \middle| \, \mu_1 \in \spec(A_1), \mu_2 \in \spec(A_2) \right\}.
    \]
    In particular, the maximal (minimal) eigenvalue of $A$ is the sum of the maximal (minimal) eigenvalues of $A_1$ and $A_2$.
\end{lemma}
\begin{proof}
    Since $A_1$ and $A_2$ are Hermitian we can choose a basis in which they are diagonal.  Then $A_1 \otimes \Id_{W_2}$ and $\Id_{W_1} \otimes A_2$ are again diagonal matrices, with diagonal entries equal to the ones of $A_1$ and $A_2$, respectively.  Then $A$ is also diagonal, and its diagonal entries are all possible sums $\mu_1 + \mu_2$ with $\mu_i$ a diagonal entry of $A_i$.
\end{proof}

\subsection{No degeneracy locus}
As in \cite{Albesiano2024}, we can assume that the morphism $\gamma$ is in fact surjective.  Indeed, the degeneracy locus
\[
    \{ x \in X \mid \gamma_x: F_x \to G_x \text{ is not surjective} \}
\]
is the zero set of the holomorphic section $\gamma^{\wedge r_G}$ of $V := (F^*)^{\wedge r_G} \otimes \det G$, and thus it is contained in the zero set $Z$ of the section $\det (\gamma^{\wedge r_G})$ of the holomorphic line bundle $\det V$.  Assuming that the \mainthmref holds for the Stein manifold $X \setminus Z$, we obtain a holomorphic section $f_0$ of $(F \otimes K_X)|_{X \setminus Z}$ such that $g = \gamma f_0$ and 
\[
    \int_{X \setminus Z} \frac{h_F(f_0,\bar{f_0})}{(h_F^* \otimes h_G)(\gamma,\bar{\gamma})^{\alpha(r_F r_G - 1)}} \leq r_F \frac{\alpha}{\alpha-1} \norm{g}_G^2 < +\infty.
\]
Since $(h_F^* \otimes h_G)(\gamma,\bar{\gamma})$ is bounded on any bounded chart, by Riemann's theorem on removable singularities it follows that $f_0$ extends across $Z$ to a holomorphic section $f$ of $F \otimes K_X$ with $\norm{f}_G^2 = \norm{f_0}_G^2$.  One has $\gamma f = g$ because $\gamma f$ and $g$ coincide on the open set $X \setminus Z$.

\begin{remark}
    The same argument proves the \mainthmref when $X$ is essentially Stein, provided it has been proved for Stein manifolds.
\end{remark}

From now on we then always assume that $X$ is Stein and $F \overset{\gamma}{\to} G$ is surjective.

\subsection{Existence of a finite-norm solution, and the dual problem}
We can reduce $X$ to a bounded pseudoconvex domain inside some larger Stein manifold $Y$, where $F,G,\gamma,g$ extend holomorphically and $\mathfrak{g},h_G,h_F$ extend smoothly.  Because the constant $r_F \frac{\alpha}{\alpha-1}$ in the bound of the \mainthmref does not depend on anything except for the rank of $F$ and the fixed constant $\alpha$, if the \mainthmref is proved under these hypotheses then the general case is obtained by standard weak-$*$ compactness arguments.

\begin{lemma}\label{prop:existence}
    Let $g$ be a holomorphic section of $G \otimes K_X$.  There exists a holomorphic section $\tilde{f}$ of $F \otimes K_X$ such that $\gamma \tilde{f} = g$ and $\norm{\tilde{f}}_F^2 < +\infty$.
\end{lemma}
\begin{proof}
    Since $X$ is relatively compact in $Y$, any section $f$ of $F \otimes K_X$ on $Y$ with $\gamma f = g$ automatically has finite $L^2$ norm when restricted to $X$.  Therefore it suffices to show that there is a solution on $Y$.

    By the surjectivity of $\gamma$ we have the short exact sequence of holomorphic vector bundles 
    \[
        0 \longto \ker\gamma \longto F \otimes K_X \overset{\gamma}{\longto} G \otimes K_X \longto 0.
    \]
    The induced long exact sequence in cohomology then gives 
    \[
        H^0(Y, F \otimes K_X) \longto H^0(Y, G \otimes K_X) \longto H^1(Y, \ker\gamma) = 0,
    \]
    where the last term vanishes by Cartan's Theorem B because $\ker\gamma$ is a holomorphic vector bundle over the Stein manifold $Y$.
\end{proof}

Naturally we have no control on the norm of this solution $\tilde{f}$, which could blow up as $X$ exhaust the original Stein manifold we started with.  Nonetheless, because there is a solution $\tilde{f}$ with finite $L^2$ norm, it follows that there is a (unique) solution $f$ with minimal $L^2$ norm.  In order to prove the \mainthmref we then need to estimate $\norm{f}_F^2$.

\begin{lemma}\label{lemma:dualization}
    Let $\tilde{f}$ be any solution to $\gamma \tilde{f} = g$ with $\norm{\tilde{f}}_F^2 < +\infty$.  Then the norm of the minimal-$L^2$-norm solution $f$ is 
    \[
        \norm{f}_F^2 = \sup_{\eta \in C^\infty_c(X, G \otimes K_X)} \frac{|\xi_\eta(\tilde{f})|^2}{\norm{\xi_\eta}_{F,*}^2},
    \]
    where the supremum runs over all smooth compactly supported sections $\eta$ of $G \otimes K_X$, and
    \[
        \xi_\eta(s) := \int_X \frac{h_G(\gamma(s),\bar{\eta})}{(h_F^* \otimes h_G)(\gamma,\bar{\gamma})^{\alpha(r_F r_G - 1) + 1}}.
    \]
\end{lemma}
The proof is an almost direct rewriting of the proof of Lemma 5.2 in \cite{Albesiano2024}.  We reproduce it here for completeness.
\begin{proof}
    We first show that 
    \begin{equation}\label{eq:annNorm}
        \norm{f}_F^2 = \sup_{\xi \in \Ann H^0(X,\ker\gamma)} \frac{|\xi(\tilde{f})|^2}{\norm{\xi}_{F,*}^2},
    \end{equation}
    where $\Ann H^0(X,\ker\gamma)$ is the space of linear functionals over $L^2$ holomorphic sections of $F \otimes K_X$ that vanish on sections of $\ker\gamma$ (seen as sections of $F \otimes K_X$).  Notice that $\xi(\tilde{f})$ is then independent of the choice of solution $\tilde{f}$ to $\gamma \tilde{f} = g$.

    Since $f$ is the minimal norm solution, for any holomorphic section $k$ of $\ker\gamma$ and $\varepsilon \in \mathbb{C}$ the section $f + \varepsilon k$ is a solution and the function 
    \[
        \mathbb{C} \ni \varepsilon \longmapsto \norm{f + \varepsilon k}_F^2 = \norm{f}_F^2 + 2\Re\left[ \varepsilon (f,k)_F \right] + O(|\varepsilon|^2)
    \]
    has minimum at $\varepsilon = 0$.  Hence $(f,k)_F = 0$, so that $f$ is perpendicular to the image of $H^0(X,\ker\gamma)$ in $H^0(X, F \otimes K_X)$, proving \eqref{eq:annNorm}.

    To conclude, note that clearly $\xi_\eta \in \Ann H^0(X,\ker\gamma)$, and that conversely if $\xi_\eta(s) = 0$ for all smooth compactly supported $\eta$ then $\gamma(s) = 0$.  Therefore the $\xi_\eta$'s are dense in $\Ann H^0(X,\ker\gamma)$, and we may thus restrict to such elements in \eqref{eq:annNorm}.
\end{proof}

By \autoref{lemma:dualization} and Cauchy--Schwarz we then have 
\begin{equation}\label{eq:dualProblem}
    \norm{f}_F^2 \leq \norm{g}_G^2 \sup_{\eta \in C^\infty_c(X, G \otimes K_X)} \frac{\norm{\mathcal{P}\eta}_G^2}{\norm{\xi_\eta}_{F,*}^2},
\end{equation}
where $\mathcal{P}$ is the Bergman projection to $L^2$-integrable (with respect to $\norm{\cdot}_G^2$) holomorphic sections of $G \otimes K_X$.  Hence the \mainthmref reduces to proving 
\[
    \norm{\mathcal{P}\eta}_G^2 \leq r_F \frac{\alpha}{\alpha-1} \norm{\xi_\eta}_{F,*}^2
\]
for all smooth and compactly supported sections $\eta$ of $G \otimes K_X$.

\section{The family of norms}\label{sec:family}
Denote by $\pr_X: P(F^* \otimes G) \to X$ the bundle projection, set $r := \rk(F^* \otimes G)= r_F r_G$, and set $h := h_F^* \otimes h_G$.  We work with the vector bundle 
\[
    V := \pr_X^* (G \otimes K_X) \otimes \mathcal{O}(1) \longto P(F^* \otimes G),
\]
endowed with the family of metrics $\mathfrak{h}_\tau$ defined below and parametrized by 
\[
    \tau \in \mathbb{L} := \{ \Re z \leq 0 \}.
\]
For a section $\sigma$ of $V \to P(F^* \otimes G)$ set
\[
    \mathfrak{h}_\tau(\sigma,\bar{\sigma}) := \frac{r_F}{V_{\Re\tau}} \frac{h_G(\sigma_v, \bar{\sigma}_v)}{h(v,\bar{v})} \left( h(v,\bar{v}) \e^{-\chi_\tau} \right)^{\alpha(r-1)},
\]
where $V_t = \frac{\pi^{r-1}}{(r-1)!} \e^{(r-1)t}$ is the volume of the ball of radius $\e^{t/2}$ in $\mathbb{R}^{2r-2}$ and
\[
    \chi_\tau := \max\left( \log\left[h(\gamma,\bar{\gamma}) h(v,\bar{v}) - |h(\gamma,\bar{v})|^2\right] - \Re\tau, \, \log\left[h(\gamma,\bar{\gamma})h(v,\bar{v})\right] \right).
\]
Notice that whether the maximum is attained by the left-hand side or the right-hand side does not depend on the choice of representative of $[v] \in P(F^* \otimes G)$, and that the weight $h(v,\bar{v}) \e^{-\chi_\tau}$ is indeed a well-defined function on $P(F^* \otimes G)$.

A section $s$ of $F \to X$ lifts to the section $\hat{s}(v) := vs$ of $V \to P(F^* \otimes G)$ (in the same way as section of a vector bundle $E \to Y$ are lifted to sections of $\mathcal{O}(1) \to P(E^*)$).

\begin{remark}
    We think of $P(F^* \otimes G)$ as parametrizing all pointwise morphisms of $F$ to $G$ up to scaling.  Notice that sections of $\mathcal{O}(1)|_{\pr_X^{-1}(x)} \to P(F^* \otimes G)|_{\pr_X^{-1}(x)}$ are identified with vectors in $(F \otimes G^*)_x$.
\end{remark}

In order to define a family of $L^2$ norms induced by $\mathfrak{h}_\tau$ we need to introduce a Fubini--Study volume form on the fibers of $P(F^* \otimes G) \to X$.  Since
\[
    \omega := \I\ddbar\log h(v,\bar{v})
\]
is a positive $(1,1)$-form when restricted (via the pullback of the inclusion) to each fiber of $P(F^* \otimes G) \to X$ we can use such $(1,1)$ form to define the volume form on each fiber as $(\iota_x^*\omega)^{\wedge (r - 1)}$, where $\iota_x: P(F^* \otimes G)_x \hookrightarrow P(F^* \otimes G)$ is the inclusion of the fiber in the total space of the fibration.%

We then have an induced family of norms given by
\[
    \norm{\sigma}_\tau^2 := \frac{r_F}{V_{\Re\tau}} \int_{[v] \in P(F^* \otimes G)} \frac{h_G(\sigma_v, \bar{\sigma}_v)}{h(v,\bar{v})} \left( h(v,\bar{v}) \e^{-\chi_\tau} \right)^{\alpha(r-1)} \wedge (\iota_x^*\omega)^{\wedge (r - 1)}
\]
(for ease of notation, here and in the following $x = \pr_X(v)$).

Since $\chi_\tau$, $\mathfrak{h}_\tau$, and $\norm{\cdot}_\tau$ clearly only depend on $t = \Re\tau$, in the following we will write $\chi_t$, $\mathfrak{h}_t$, and $\norm{\cdot}_t$, respectively, with the understanding that $t \in (-\infty,0]$.

We now proceed to study the extrema of the family of norms.  For $t = 0$ we have
\[
    \chi_0 = \log\left[h(\gamma,\bar{\gamma})h(v,\bar{v})\right]
\]
and so 
\[
    \norm{\hat{s}}_0^2 = \frac{r_F}{V_0} \int_{[v] \in P(F^* \otimes G)} \frac{h_G(vs, \overline{vs})}{h(v,\bar{v}) h(\gamma,\bar\gamma)^{\alpha(r-1)}} \wedge (\iota_x^*\omega)^{\wedge (r - 1)}.
\]
Fix $x \in X$ and choose local coordinates $e_F^1, \dots, e_F^{r_F}$ and $e_G^1, \dots, e_G^{r_G}$ for the vector spaces $F_x$ and $G_x$, respectively, so that $(h_F)_x$ and $(h_G)_x$ are given by the identity matrix.  Write $s = s_j e_F^j$ and $v = v^j_k e^*_{F,j} e_G^k$.  Then $vs = v^j_k s_j e_G^k$ and the integral on $P(F^* \otimes G)_x$ becomes
\[
    \frac{1}{h(\gamma,\bar\gamma)^{\alpha(r-1)}} \int_{[\{v^j_k\}] \in \mathbb{P}_{r-1}} \frac{\sum_{k=1}^{r_G}|\sum_{j=1}^{r_F} v^j_k s_j|^2}{|v|^2} \wedge \dVFS.
\]
By symmetry we have 
\[
\begin{split}
    &\int_{[\{v^j_k\}] \in \mathbb{P}_{r-1}} \frac{\sum_{k=1}^{r_G}|\sum_{j=1}^{r_F} v^j_k s_j|^2}{|v|^2} \wedge \dVFS \\
    &\quad= r_G \int_{[\{v^j_k\}] \in \mathbb{P}_{r-1}} \frac{|\sum_{j=1}^{r_F} v^j_k s_j|^2}{|v|^2} \wedge \dVFS = r_G \int_{[\{v^j_k\}] \in \mathbb{P}_{r-1}} \frac{|v \cdot \tilde{s}|^2}{|v|^2} \wedge \dVFS,
\end{split}
\]
where $\tilde{s}$ is the vector whose first $r_F$-entries are given by $s$, and the remaining $r - r_F = (r_G - 1)r_F$ are zero.  As in \cite{Albesiano2024}, a computation in polar coordinates then gives
\[
    r_G \int_{[\{v^j_k\}] \in \mathbb{P}_{r-1}} \frac{|v \cdot \tilde{s}|^2}{|v|^2} \wedge \dVFS = \frac{\pi^{r-1}}{r!} r_G |\tilde{s}|^2 = \frac{V_0}{r_F} h_F(s,\bar{s}).
\]
Hence we conclude that
\[
    \norm{\hat{s}}_0^2 = \int_X \frac{h_F(s,\bar{s})}{h(\gamma,\bar\gamma)^{\alpha(r-1)}} = \norm{s}_F^2.
\]

Next we study what happens in the limit $t \to -\infty$.  For this purpose, fix $x \in X$ and $t<0$, and consider the subset $A_{t,x}$ of the fiber $P(F^* \otimes G)_x$ given by the region where the maximum is attained by the right-hand side, namely:
\[
    A_{t,x} := \left\{ v \in P(F^* \otimes G)_x \, \middle| \, 1 - \frac{|h(\gamma,\bar{v})|^2}{h(\gamma,\bar\gamma)h(v,\bar{v})} \leq \e^t \right\}.
\]
For each fixed $x$ we choose homogeneous coordinates $[v_1:\dots:v_r]$ for $P(F^* \otimes G)_x$ so that $v_1$ is parallel to $\gamma(x)$ in $(F^* \otimes G)_x$ (not vanishing by assumption).  Then in the standard local coordinates for the chart $v_1 \neq 0$ one sees that $A_{t,x}$ is a ball of real dimension $2r - 2$, centered at $[h(x)]$, and with radius asymptotic to $\e^{t/2}$ when $t \to -\infty$.  Notice that the volume of $A_{t,x}$ is then asymptotic to $V_t = \frac{\pi^{r-1}}{(r-1)!} \e^{(r-1)t}$.

Let $A_t$ be the subset of $P(F^* \otimes G)$ given by the union of all $A_{t,x}$ for $x \in X$. We then write
\[
    \norm{\hat{s}}_t^2 = \one_t + \two_t,
\]
with
\[
    \one_t := r_F \fint_{A_t} \frac{h_G(vs, \overline{vs})}{h(v,\bar{v}) h(\gamma,\bar\gamma)^{\alpha(r-1)}} \wedge (\iota_x^*\omega)^{\wedge (r - 1)}
\]
and
\[
    \two_t := \frac{r_F}{V_t} \int\limits_{x \in X} \int\limits_{[v] \in P(F^* \otimes G)_x \setminus A_{t,x}} \frac{h_G(vs, \overline{vs}) \e^{\alpha(r-1)t}}{h(v,\bar{v}) h(\gamma,\bar\gamma)^{\alpha(r-1)}} \wedge \frac{(\iota_x^*\omega)^{\wedge (r - 1)}}{\left( 1 - \frac{|h(\gamma,\bar{v})|^2}{h(\gamma,\bar{\gamma})h(v,\bar{v})}\right)^{\alpha(r-1)}}.
\]

Integrating along the fibers $P(F^* \otimes G)_x$ first one obtains
\begin{equation}\label{eq:It}
    \lim_{t \to -\infty} \one_t = r_F \int_X \frac{h_G(\gamma s, \overline{\gamma s})}{h(\gamma,\bar{\gamma})^{\alpha(r-1)+1}}
\end{equation}
since the integral along each fiber is the average over the ball $A_{t,x}$ as the radius of the ball goes to zero.  The second term can be rewritten as 
\[
    \two_t = \e^{-(r-1)t} \int_t^0 \e^{-\alpha(r-1)(\tilde{t}-t)} \dif\nu_s(\tilde{t})
\]
with $\nu_s(t) = \e^{(r-1)t} \one_t$.

We now recall the following calculus lemma (see \cite[Section 3]{Albesiano2024} for the proof).
\begin{lemma}
    Let $\nu: (-\infty,0] \rightarrow \mathbb{R}_+$ be an absolutely continuous increasing function such that
    \[
      \lim_{t \rightarrow -\infty} \e^{-Bt} \nu(t) = A < +\infty
    \]
    for some $B > 0$. Then
    \[
      \lim_{t \rightarrow -\infty} \e^{-Bt} \int_t^0 \e^{-p(s-t)} \dif \nu(s) = \frac{AB}{p-B}
    \]
    for all $p>B$.
\end{lemma}

Hence
\[
    \lim_{t \to -\infty} \two_t = \frac{r_F}{\alpha-1} \int_X \frac{h_G(\gamma s, \overline{\gamma s})}{h(\gamma,\bar{\gamma})^{\alpha(r-1)+1}}
\]
so that overall 
\begin{equation}\label{eq:limit}
    \lim_{t \to -\infty} \norm{\hat{s}}_t^2 = r_F \frac{\alpha}{\alpha-1} \int_X \frac{h_G(\gamma s, \overline{\gamma s})}{h(\gamma,\bar{\gamma})^{\alpha(r-1)+1}} = r_F \frac{\alpha}{\alpha-1} \norm{\gamma s}_G^2,
\end{equation}
retrieving a multiple of the norm squared of the image of $s$ under $\gamma$.

\section{Curvature}\label{sec:curvature}
In order to apply the Berndtsson--Lempert machinery we see the family of metrics $\mathfrak{h}_\tau$ as a single metric $\mathfrak{h}$ for 
\[
    \pr_{P(F^* \otimes G)}^* V \longto P(F^* \otimes G) \times \mathbb{L}
\]
defined by $\mathfrak{h}|_{P(F^* \otimes G) \times \{\tau\}} := \mathfrak{h}_\tau$.  The degeneration argument that will ultimately give us the estimate in the \mainthmref is deferred to \autoref{sec:convexity}.  In this section we only check that the curvature of $\mathfrak{h}$ satisfies the hypotheses of D.~Varolin's vector bundle version of Berndtsson's theorem \cite[Theorem 2]{Varolin2025}, i.e.~that the curvature of $\mathfrak{h}$ plus the curvature induced by $(\iota_x^* \omega)^{\wedge (r - 1)}$ is non-negative in the sense of Griffiths on the total space $P(F^* \otimes G) \times \mathbb{L}$, and non-negative in the sense of Nakano on each fiber $P(F^* \otimes G) \times \{\tau\}$.  

We compute the curvature of $\mathfrak{h}$ to be 
\[
    \Theta_{\mathfrak{h}} = \Theta_{h_G} + \alpha(r-1)\I\ddbar\chi + (1 - \alpha(r-1))\I\ddbar\log h(v,\bar{v}).
\]

To compute the curvature of the volume form $(\iota_x^*\omega)^{\wedge (r-1)}$ at a point $([v],\tau) \in P(F^* \otimes G) \times \mathbb{L}$, we choose a holomorphic frame $\epsilon^1, \dots, \epsilon^r$ for $F^* \otimes G$ on an open neighborhood $U \subset X$ of $\pr_X([v])$.  This fixes homogeneous coordinates on $P(F^* \otimes G)|_U$: vectors in $(F^* \otimes G)_x$ ($x \in U$) can be written as $v = v_1 \epsilon^1_x + \dots + v_r \epsilon^r_x$ and then $[v_1 : \dots : v_r]$ are the homogeneous coordinates of $[v] \in P(F^* \otimes G)_x$.  Let $H^{j\bar{k}}_x := h(\epsilon^j_x, \bar\epsilon^k_x)$, and choose Cartesian coordinates $z_j := v_j/v_r$ in the chart $\{v_r \neq 0\}$ of $P(F^* \otimes G)|_U$.  Then 
\[
\begin{split}
    g^{j\bar{k}} :=& \, \partial_{z_j}\partial_{\bar{z}_k} \log h(v,\bar{v}) \\
    =& \, \partial_{z_j}\partial_{\bar{z}_k} \log \left( H^{\ell\bar{m}} z_\ell \bar{z}_m  + H^{r\bar{m}} \bar{z}_m + H^{\ell\bar{r}} z_\ell + H^{r\bar{r}} \right) \\
    =& \, \partial_{z_j} \frac{H^{\ell\bar{k}}z_\ell + H^{r\bar{k}}}{H^{\ell\bar{m}} z_\ell \bar{z}_m  + H^{r\bar{m}} \bar{z}_m + H^{\ell\bar{r}} z_\ell + H^{r\bar{r}}} \\
    =& \, \frac{H^{j\bar{k}}}{H^{\ell\bar{m}} z_\ell \bar{z}_m  + H^{r\bar{m}} \bar{z}_m + H^{\ell\bar{r}} z_\ell + H^{r\bar{r}}} - \frac{(H^{\ell\bar{k}}z_\ell + H^{r\bar{k}}) (H^{j\bar{m}}\bar{z}_m + H^{j\bar{r}})}{(H^{\ell\bar{m}} z_\ell \bar{z}_m  + H^{r\bar{m}} \bar{z}_m + H^{\ell\bar{r}} z_\ell + H^{r\bar{r}})^2}.
\end{split}
\]
Write
\[
H = \begin{pmatrix}
    A & b \\ b^\dagger & H^{r\bar{r}}
\end{pmatrix},
\]
where $A$ is the submatrix of the first $(r-1) \times (r-1)$ entries and $b$ is the vector of the first $r-1$ entries of the last column.  Then
\[
    g = \frac{A - ww^\dagger}{(z,1)^\dagger H(z,1)} \quad \text{with } w := \frac{Az + b}{\sqrt{(z,1)^\dagger H(z,1)}}
\]
By the matrix determinant lemma we have
\[
    \det g = \frac{1 - w^\dagger A^{-1} w}{((z,1)^\dagger H(z,1))^{r-1}} \det A.
\]
Since 
\[
    1 - w^\dagger A^{-1} w = 1 - \frac{z^\dagger A z + z^\dagger b + b^\dagger z + b^\dagger A^{-1} b}{(z,1)^\dagger H(z,1)} = \frac{H^{r\bar{r}} - b^\dagger A^{-1}b}{(z,1)^\dagger H (z,1)},
\]
we conclude by Schur's formula that 
\[
    \det g = \frac{\det H}{((z,1)^\dagger H(z,1))^r},
\]
and so the volume form in the chart $\{v_r \neq 0\}$ is
\[
    (\iota_x^* \omega)^{\wedge(r-1)} = \left(\frac{\I}{2}\right)^{r-1} (r-1)! \frac{\det H}{((z,1)^\dagger H(z,1))^r} \dif z \wedge \dif\bar{z}.
\]
Therefore the curvature introduced by the volume form is
\[
    - \I\ddbar\log\det h + r \I\ddbar\log h(v,\bar{v}) = \Theta_{\det h} + r \I\ddbar\log h(v,\bar{v}).
\]
Recalling that $h = h_F^* \otimes h_G$ and that $\det(h_F^* \otimes h_G) = (\det h_F^*)^{\otimes r_G} \otimes (\det h_G)^{\otimes r_F}$, we notice that
\[
    \Theta_{\det h} = r_F \Theta_{\det h_G} - r_G \Theta_{\det h_F}.
\]

In order to satisfy the curvature conditions of Berndtsson--Varolin's theorem it then suffices to show that 
\begin{equation}\label{eq:curvNak}
    \Theta_{(h_G \otimes \det h_G) \e^{-\eta}} \geqNak 0,
\end{equation}
with
\[
\begin{split}
    \eta :=& \alpha(r-1)\chi + (r+1 - \alpha(r-1)) \I\log h(v,\bar{v}) \\
    &\quad- (r_F - 1) \log\det h_G + r_G \log\det h_F.
\end{split}
\]

We now recall the following theorem of J.-P.~Demailly and H.~Skoda \cite{DemaillySkoda1980}.
\begin{theorem*}
    Let $E \to X$ be a holomorphic vector bundle, and let $h$ be a Hermitian metric with non-negative curvature in the sense of Griffiths.  Then the metric $h \otimes \det h$ for the vector bundle $E \otimes \det E \to X$ is non-negative in the sense of Nakano.
\end{theorem*}
By Demailly--Skoda's theorem, to show \eqref{eq:curvNak} it suffices to show that $h_G \e^{-\frac{\eta}{r_G +1}}$ has non-negative curvature in the sense of Griffiths, i.e.
\begin{equation}\label{eq:curvGrif}
    \Theta_{h_G} + \frac{1}{r_G + 1} \I\ddbar\eta \geqGrif 0.
\end{equation}

Recall that $\chi = \max\left(\chi^{(1)} - \Re\tau, \chi^{(2)}\right)$, where 
\[
\begin{split}
    \chi^{(1)} &= \log\left[ h(\gamma,\bar\gamma) h(v,\bar{v}) - |h(\gamma,\bar{v})|^2\right] \\
    &= \log\left[ h^{\otimes2}\left(\gamma \otimes v - v \otimes \gamma, \overline{\gamma \otimes v - v \otimes \gamma}\right) \right] - \log 2
\end{split} 
\]
and
\[
    \chi^{(2)} = \log h(\gamma,\bar\gamma) + \log h(v,\bar{v}),
\]
so condition \eqref{eq:curvGrif} is satisfied when
\begin{equation}\label{eq:curvGrif1}
\begin{split}
    (r_G &+ 1)\Theta_{h_G} + \alpha(r-1)\I\ddbar\log\left[ h^{\otimes2}\left(\gamma \otimes v - v \otimes \gamma, \overline{\gamma \otimes v - v \otimes \gamma}\right) \right] \\
    &+ (r+1 - \alpha(r-1)) \I\ddbar\log h(v,\bar{v}) \\
    &+ (r_F - 1) \Theta_{\det h_G} - r_G \Theta_{\det h_F} \geqGrif 0
\end{split}
\end{equation}
and
\begin{equation}\label{eq:curvGrif2}
\begin{split}
    (r_G &+ 1)\Theta_{h_G} + \alpha(r-1)\I\ddbar\log h(\gamma,\bar{\gamma}) \\
    &+ (r+1) \I\ddbar\log h(v,\bar{v}) \\
    &+ (r_F - 1) \Theta_{\det h_G} - r_G \Theta_{\det h_F} \geqGrif 0.
\end{split}
\end{equation}

\begin{remark}
    Notice that by construction we have
    \[
        h^{\otimes2}\left(\gamma \otimes v - v \otimes \gamma, \overline{\gamma \otimes v - v \otimes \gamma}\right) \geq 2\e^{\Re\tau} h(\gamma,\bar{\gamma}) h(v,\bar{v})
    \]
    in the region where $\chi = \chi^{(1)} - \Re\tau$.  Hence in the computations of curvature in \eqref{eq:curvGrif1} we can always assume that $[v] \neq [\gamma]$.
\end{remark}

We now study conditions \eqref{eq:curvGrif1} and \eqref{eq:curvGrif2} more carefully.  To start we claim that we can reduce these curvature conditions on $P(F^* \otimes G)$ with an analogous curvature condition on $X$ by replacing $v$ by a section $w$ of $F^* \otimes G \to X$.  Indeed, Griffiths positivity can be checked on germs of curves in $P(F^* \otimes G)$.  If the germ project to zero on $X$ we can just restrict to a curve $C$ in the fiber.  In such case \eqref{eq:curvGrif2} reduces to
\[
    \I\ddbar \log h(v,\bar{v})|_C \geq 0,
\]
which is automatically satisfied, and \eqref{eq:curvGrif1} then reduces to
\[
    \I\ddbar\log\left[ h^{\otimes2}\left(\gamma \otimes v - v \otimes \gamma, \overline{\gamma \otimes v - v \otimes \gamma}\right) \right]\bigg|_C \geq 0,
\]
also automatically satisfied since on $C$ we have
\[
\begin{split}
    &\I\ddbar\log\left[ h^{\otimes2}\left(\gamma \otimes v - v \otimes \gamma, \overline{\gamma \otimes v - v \otimes \gamma}\right) \right] \\
    &\quad= \I\partial\frac{h^{\otimes2}(\gamma \otimes v - v \otimes \gamma, \bar{\gamma} \otimes \dif\bar{v} - \dif\bar{v} \otimes \gamma)}{h^{\otimes2}\left(\gamma \otimes v - v \otimes \gamma, \overline{\gamma \otimes v - v \otimes \gamma}\right)} \\
    &\quad= \frac{\I}{\left|\gamma \otimes v - v \otimes \gamma\right|_{h^{\otimes2}}^4} \begin{bmatrix}\left|\gamma \otimes \dif v - \dif v \otimes \gamma\right|_{h^{\otimes2}}^2 \left|\gamma \otimes v - v \otimes \gamma\right|_{h^{\otimes2}}^2 \qquad\quad \\ \quad\qquad- \left|h^{\otimes2}(\gamma \otimes v - v \otimes \gamma, \overline{\gamma \otimes \dif v - \dif v \otimes \gamma})\right|^2 \end{bmatrix},
\end{split}
\]
which is non-negative by the Cauchy--Schwarz inequality.

If the germ does not project to zero we can instead choose a curve $C$ contained in the image in $P(F^* \otimes G)$ of the graph of some section $w$ of $F^* \otimes G \to X$ with $[w_x] = v$ and $(\nabla^{1,0} w)_x = 0$.  Then one notices that 
\[
    \I\ddbar\log h(w,\bar{w})_x = -\frac{h(\Theta_h w_x,\bar{w}_x)}{h(w_x,\bar{w}_x)} = -\frac{h(\Theta_h v,\bar{v})}{h(v,\bar{v})}
\]
in the sense of $(1,1)$-forms.  Condition \eqref{eq:curvGrif1} then becomes
\[
\begin{split}
    (r_G &+ 1)\Theta_{h_G} + (r_F - 1) \Theta_{\det h_G} - r_G \Theta_{\det h_F} \\
    &\quad\geqGrif \alpha(r-1) \frac{h^{\otimes2}( \Theta_{h^{\otimes2}} (\gamma \otimes v - v \otimes \gamma), \overline{\gamma \otimes v - v \otimes \gamma})}{h^{\otimes 2}(\gamma \otimes v - v \otimes \gamma, \overline{\gamma \otimes v - v \otimes \gamma})} \\
    &\qquad\qquad + (r+1 - \alpha(r-1)) \frac{h(\Theta_hv,\bar{v})}{h(v,\bar{v})},
\end{split}
\]
which means
\begin{equation}\label{eq:curvEig1}
\begin{split}
    (r_G &+ 1)\frac{h_G(\Theta_{h_G}u, \bar{u})}{h_G(u,\bar{u})} + (r_F - 1) \Theta_{\det h_G} - r_G \Theta_{\det h_F} \\
    &\quad\geq \alpha(r-1) \frac{h^{\otimes2}( \Theta_{h^{\otimes2}} (\gamma \otimes v - v \otimes \gamma), \overline{\gamma \otimes v - v \otimes \gamma})}{h^{\otimes 2}(\gamma \otimes v - v \otimes \gamma, \overline{\gamma \otimes v - v \otimes \gamma})} \\
    &\qquad\quad + (r+1 - \alpha(r-1)) \frac{h(\Theta_hv,\bar{v})}{h(v,\bar{v})}
\end{split}
\end{equation}
in the sense of $(1,1)$-forms for all $u \in G_x$, $v \in (F^* \otimes G)_x$, and $x \in X$.

\begin{remark}\label{rmk:alphaUpperBound}
    The proof in \cite{Albesiano2024} introduced an upper bound on $\alpha$ because the extra positivity in the fiber direction coming from $\chi^{(0)}$ was given up.  Computing the curvature more carefully as above allows to remove the artificial upper bound by compensating the negative term $-\alpha(r-1)\I\ddbar\log h(v,\bar{v})$ with this extra positivity.  Note that, albeit not explicitly, the negative term in the author's previous work was already compensated by the positivity along fibers of the weight at the points where $\chi = \chi^{(2)}$.
\end{remark}

In the same way condition \eqref{eq:curvGrif2} becomes
\[
\begin{split}
    (r_G &+ 1)\Theta_{h_G} + (r_F - 1) \Theta_{\det h_G} - r_G \Theta_{\det h_F} \\
    &\quad\geqGrif \alpha(r-1) \frac{h(\Theta_h \gamma, \bar{\gamma})}{h(\gamma, \bar{\gamma})} + (r+1) \frac{h(\Theta_h v, \bar{v})}{h(v,\bar{v})},
\end{split}
\]
which means
\begin{equation}\label{eq:curvEig2}
\begin{split}
    (r_G &+ 1)\frac{h_G(\Theta_{h_G}u, \bar{u})}{h_G(u,\bar{u})} + (r_F - 1) \Theta_{\det h_G} - r_G \Theta_{\det h_F} \\
    &\quad\geq \alpha(r-1) \frac{h(\Theta_h \gamma, \bar{\gamma})}{h(\gamma, \bar{\gamma})} + (r+1) \frac{h(\Theta_h v, \bar{v})}{h(v,\bar{v})}
\end{split}
\end{equation}
in the sense of $(1,1)$-forms for all $u \in G_x$, $v \in (F^* \otimes G)_x$, and $x \in X$.

By the definitions and properties of $\lambda$ and $\Lambda$, it then suffices to show 
\[
\begin{split}
    (r_G + 1) &\lambda_{h_G} + (r_F - 1) \Tr\Theta_{h_G} - r_G \Tr\Theta_{h_F} \\
    &\geq \alpha(r-1) \Lambda_{h^{\otimes2}} + (r+1 - \alpha(r-1)) \Lambda_h = (r+1 + \alpha(r-1)) \Lambda_h
\end{split}
\]
for \eqref{eq:curvEig1}, and 
\[
\begin{split}
    (r_G + 1) \lambda_{h_G} + &(r_F - 1) \Tr\Theta_{h_G} - r_G \Tr\Theta_{h_F} \\
    &\geq \alpha(r-1) \Lambda_{h} + (r+1) \Lambda_h = (r+1 + \alpha(r-1)) \Lambda_h
\end{split}
\]
for \eqref{eq:curvEig2}.  Since 
\[
    \Lambda_h = \Lambda_{h_F^* \otimes h_G} = \Lambda_{h_F^*} + \Lambda_{h_G} = \Lambda_{h_G} - \lambda_{h_F},
\]
we conclude that the requirements on curvature are satisfied when 
\[
    (r_G + 1) \lambda_{h_G} + (r_F - 1)\Tr\Theta_{h_G} - r_G \Tr\Theta_{h_F} \geq (r+1 + \alpha(r-1)) (\Lambda_{h_G} - \lambda_{h_F}),
\]
as in the assumptions of the \mainthmref.

\section{Convexity of the degenerating family}\label{sec:convexity}
The argument to complete the proof of the \mainthmref is at this point fairly standard, and follows the lines of \cite{BerndtssonLempert2016}.

We know that the metric $\mathfrak{h}$ for $\pr_{P(F^* \otimes G)}^* V \to P(F^* \otimes G) \times \mathbb{L}$ has non-negative curvature in the sense of Nakano (\autoref{sec:curvature}), and induces a family of norms $\left\{\,\norm{\cdot}_\tau\right\}_{\tau \in \mathbb{L}}$ such that $\norm{\hat{s}}_0^2 = \norm{s}_F^2$ and $\lim_{\Re\tau \to -\infty} \norm{\hat{s}}_{\Re\tau}^2 = r_F \frac{\alpha}{\alpha-1} \norm{\gamma s}_G^2$ (\autoref{sec:family}).

For each $\tau \in \mathbb{L}$ define 
\[
    \mathcal{H}_\tau := \left\{ \sigma \in H^0(P(F^* \otimes G), V) \, \middle| \,\, \norm{\sigma}_\tau^2 < +\infty \right\}.
\]
By Varolin's vector bundle version \cite[Theorem 2]{Varolin2025} of Berndtsson's theorem \cite[Theorem 1]{Berndtsson2009}, $\mathcal{H} \to \mathbb{L}$ is a holomorphic vector bundle, and the metric induced by $\mathfrak{h}$ to $\mathcal{H} \to \mathbb{L}$ is non-negatively curved in the sense of Nakano.  In particular the map 
\[
    \mathbb{L} \ni \tau \longmapsto \log\norm{\xi}_{\tau,*}^2
\]
is subharmonic for all functionals $\xi$ on the space of $L^2$-integrable holomorphic sections of $F \to X$.  Since this map only depends on $t = \Re\tau$, it follows that the map
\[
    (-\infty,0) \ni t \longmapsto \log\norm{\xi}_{t,*}^2
\]
is convex for all $\xi$, so that in particular $t \mapsto \log\norm{\xi_\eta}_{t,*}^2$ is convex for all $\eta \in C^\infty_c(X,G \otimes K_X)$.

We now claim that $t \mapsto \log\norm{\xi_\eta}_{t,*}^2$ is bounded for all $\eta$.  Indeed
\[
    \norm{\xi_\eta}_{t,*}^2 = \sup_{\norm{\hat{s}}_t^2 = 1} \left| \int_X \frac{h_G(\gamma(s),\bar{\eta})}{h(\gamma,\bar{\gamma})^{\alpha(r - 1) + 1}} \right|^2 \leq \norm{\eta}_G^2 \sup_{\norm{\hat{s}}_t^2 = 1} \int_X \frac{h_G(\gamma(s),\overline{\gamma(s)})}{h(\gamma,\bar{\gamma})^{\alpha(r-1)+1}}.
\]
Notice that $\norm{\eta}_G^2 < +\infty$ because $\eta$ has compact support.  Moreover, by \eqref{eq:It} we have 
\[
    \int_X \frac{h_G(\gamma(s),\overline{\gamma(s)})}{h(\gamma,\bar{\gamma})^{\alpha(r-1)+1}} \leq \frac{2}{r_F} \one_t \leq \frac{2}{r_F} \norm{\hat{s}}_t^2 = \frac{2}{r_F}
\]
for all $t$ sufficiently negative.  Hence $\norm{\xi_\eta}_{t,*}^2 \leq \frac{2}{r_F} \norm{\eta}_G^2$ for all $t$ negative enough, and thus $t \mapsto \log\norm{\xi}_{t,*}^2$ is bounded for all $\eta$.

Because $(-\infty,0] \ni t \mapsto \log\norm{\xi}_{t,*}^2$ is bounded and convex, it must be increasing, so in particular
\[
    \norm{\xi_\eta}_{F,*}^2 = \norm{\xi_\eta}_{0,*}^2 \geq \lim_{t \to -\infty} \norm{\xi_\eta}_{t,*}^2
\]
for all $\eta \in C^\infty_c(X,G \otimes K_X)$.

Let now $s$ be any holomorphic section of $F \otimes K_X \to X$ with finite $L^2$ norm such that $\gamma(s) = \mathcal{P} \eta$ (we know that such section $s$ exists by \autoref{prop:existence}).  Then 
\[
    \norm{\xi_\eta}_{F,*}^2 \geq \lim_{t \to -\infty} \norm{\xi_\eta}_{t,*}^2 \geq \lim_{t \to -\infty} \frac{|\xi_\eta(s)|^2}{\norm{\hat{s}}_t^2} = \lim_{t \to -\infty} \frac{\norm{\mathcal{P}\eta}_G^4}{\norm{\hat{s}}_t^2} = \frac{\alpha-1}{r_F \alpha} \norm{\mathcal{P}\eta}_G^2,
\]
where the last equality comes from \eqref{eq:limit}.  We then conclude by \eqref{eq:dualProblem} that the minimal-norm solution $f$ has norm 
\[
    \norm{f}_F^2 \leq r_F \frac{\alpha}{\alpha-1} \norm{g}_G^2,
\]
proving the \mainthmref.

\end{document}